\documentclass[a4paper]{article} 
\usepackage[T1]{fontenc} 
\usepackage{lmodern} 
\usepackage[english]{babel} 
\usepackage{a4wide} 
\usepackage{amsmath,amssymb,amsthm} 
\usepackage{mathrsfs} 
\usepackage{mathtools} 
\usepackage{ifpdf} 
\usepackage[numbers,sort&compress]{natbib} 
\usepackage{graphicx} 

\ifpdf%
\usepackage[unicode,bookmarks=false,
hyperindex]{hyperref}
\hypersetup{pdfstartview=FitH,
  linkbordercolor=[rgb]{0.5, 1, 1},
  citebordercolor=[rgb]{0.5, 1, 0.5},
hypertexnames=false}
\else 
\fi 

\newtheorem{theorem}{Theorem} 
\newtheorem{lemma}{Lemma} 
\newtheorem{proposition}{Proposition} 

\theoremstyle{definition} 
\newtheorem{problem}{Question} 
\newtheorem{altproblem}{Question} 
\newtheorem{remark}{Remark} 
 
\newtheorem{example}{Example} 

\newcommand\setA{\mathscr{A}} 
\newcommand\setB{\mathscr{B}} 
 
\newcommand\setM{\mathscr{M}} 
\newcommand\setH{\mathscr{H}} 
\newcommand\setF{\mathscr{F}} 
\newcommand\bbN{\mathbb{N}} 

\title{On the boundedness of infinite matrix products with alternating
factors from two sets of matrices}

\author{Victor Kozyakin\thanks{Kharkevich Institute for Information 
    Transmission Problems, Russian Academy of Sciences, Bolshoj
    Karetny lane 19,
Moscow 127051, Russia, e-mail: kozyakin@iitp.ru}} 

\begin{document} 
\date{} 
\maketitle 

\begin{abstract} 
  We consider the question of boundedness of matrix products 
  $A_{n}B_{n}\cdots A_{1}B_{1}$ with factors from two sets of matrices, 
  $A_{i}\in\setA$ and $B_{i}\in\setB$, due to an appropriate choice of 
  matrices $\{B_{i}\}$. It is assumed that for every sequence of matrices 
  $\{A_{i}\}$ there is a sequence of matrices $\{B_{i}\}$ for which the 
  sequence of matrix products $\{A_{n}B_{n}\cdots 
  A_{1}B_{1}\}_{n=1}^{\infty}$ is norm-bounded. Some situations are 
  described where in this case the norms of the matrix products 
  $A_{n}B_{n}\cdots A_{1}B_{1}$ are uniformly bounded, i.e., 
  $\|A_{n}B_{n}\cdots A_{1}B_{1}\|\le C$ for all natural numbers $n$, where 
  $C > 0$ is a constant independent of the sequence $\{A_{i}\}$ and the 
  corresponding sequence $\{B_{i}\}$. For the general case, the question of 
  the validity of the corresponding statement remains open. 
  \medskip 

  \noindent\textbf{Key words and phrases:} infinite matrix products, 
  alternating factors, boundedness of matrix products 
  \medskip 

  \noindent\textbf{2020 Mathematics Subject Classification:} 40A20; 39A22; 
  93B05; 93C55 
\end{abstract} 

\section{Introduction}\label{S:intro} 

Let $\setM$ be a set of matrices. By $\setM^{\bbN}$ we denote the set of 
infinite sequences of matrices $\{M_{i}\in\setM\colon i=1,2,\ldots\}$ and by 
$\setM^{n}$, where $n=1,2,\ldots\,$, the set of finite ordered sets 
$\{M_{i}\in\setM\colon i=1,2,\ldots,n\}$ consisting of $n$ matrices. By 
$\setM(p,q)$ we denote the space of real $(p\times q)$-matrices with the 
topology of element-wise convergence. 

In the theory of matrix products the following statement is well known (see 
  for example,~\cite{Koz:AiT90:10:e, DaubLag:LAA92, BerWang:LAA92, Gurv:LAA95,
Hartfiel:02}): 
\begin{proposition}\label{St:1} 
  Let $\setA\subset\setM(N,N)$ be a finite set of matrices in which for each 
  sequence of matrices $\{A_{n}\}\in\setA^{\bbN}$ the sequence of norms 
  $\{\|A_{n}\cdots A_{1}\|\}_{n=1}^{\infty}$ is bounded. Then all sequences of
  norms $\{\|A_{n}\cdots A_{1}\|\}_{n=1}^{\infty}$ are uniformly bounded,
  i.e., there is a constant $C > 0$ such that for any sequence of matrices
  $\{A_{n}\}\in\setA^{\bbN}$ for all $n=1,2,\ldots$ the inequalities
  $\|A_{n}\cdots A_{1}\|\le C$ hold. 
\end{proposition} 

Some time ago an attempt was made in~\cite{Koz:DDNS18} to extend the analysis 
of convergence of matrix products to the matrix products with alternating 
factors from two sets of matrices. The need for such an analysis was 
motivated by the problem of stabilization of switching linear systems with 
discrete time, see, e.g.,~\cite{Stanford:SIAMJCO79, Koz:AiT90:10:e, SunGe05, 
JunMas:SIAMJCO17, LinAnt:IEEETAC09}. More specifically, in~\cite{Koz:DDNS18} 
the following question was considered: Let $\setA\subset\setM (N,M)$ and 
$\setB\subset\setM (M,N)$ be finite sets of matrices, what can be said about 
the convergence rate to zero of the matrix products 
\[ 
  A_{n}B_{n}\cdots A_{1}B_{1},\qquad A_{i}\in\setA,~B_{i}\in\setB, 
\] 
under the condition that convergence to zero $A_{n}B_{n}\cdots 
A_{1}B_{1}\to0$ for any sequence of matrices $\{A_{i}\}$ can be guaranteed by 
a suitable choice of a sequence of matrices $\{B_{i}\}$? As it turns out 
(see~{\cite[Theorem~2]{Koz:DDNS18}}), the following statement is true: 
\begin{proposition}\label{T:main} 
  Let $\setA\subset\setM(N,M)$ and $\setB\subset\setM (M,N)$ be sets of 
  matrices with the property, that for any sequence of matrices 
  $\{A_{n}\}\in\setA^{\bbN}$, one can specify a sequence of matrices 
  $\{B_{n}\}\in\setB^{\bbN}$ such that $A_{n}B_{n}\cdots A_{1}B_{1}\to0$. 
  Then there are constants $C > 0$ and $\lambda\in(0,1)$ such that for every 
  sequence of matrices $\{A_{n}\}\in\setA^{\bbN}$ there is a sequence of 
  matrices $\{B_{n}\}\in\setB^{\bbN}$, for which $\|A_{n}B_{n}\cdots 
  A_{1}B_{1}\|\le C\lambda^{n}$ for $n=1,2,\ldots\,$. 
\end{proposition} 

The aim of this article is to continue the study of the properties of matrix 
products with alternating factors from two sets of matrices. We say that the 
sequence of matrices $\{A_{n}\}\in\setA^{\bbN}$ is 
\emph{$\setB$-right-bounded}\footnote{Here the prefix `right-' means that in 
  the corresponding matrix products the factors $B_{i} $ are applied to the
factors $A_{i}$ on the right-hand side.} if there is a constant 
$C=C(\{A_{n}\})$ such that 
\begin{equation}\label{E:defABbounded} 
  \exists~\{B_{n}\}\in\setB^{\bbN}:\quad \|A_{n}B_{n}\cdots A_{1}B_{1}\|\le 
  C, \quad n=1,2,\ldots\,, 
\end{equation} 
i.e., if there is a sequence of matrices $\{B_{n}\}\in\setB^{\bbN}$ for which 
the sequence of norms 
\[ 
  \{\|A_{n}B_{n}\cdots A_{1}B_{1}\|\}_{n=1}^{\infty} 
\] 
is bounded. We will be interested in whether Proposition~\ref{St:1} (or any 
of its analogs) is applicable to matrix products with alternating factors 
from two sets of matrices, apparently first formulated in~\cite[Question 
2]{Koz:DDNS18}: 

\begin{problem}\label{Q:1} 
  Let $\setA$ and $\setB$ be finite sets of matrices such that every sequence 
  of matrices ${\{A_{n}\}\in\setA^{\bbN}}$ is $\setB$-right-bounded. In this
  case, is there a universal constant $C > 0$ such that for any sequence of 
  matrices $\{A_{n}\}\in\setA^{\bbN}$ inequalities~\eqref{E:defABbounded} 
  hold? 
\end{problem} 

Question~\ref{Q:1} can be reformulated in a slightly different form that is 
more suitable for further consideration. 

Given a natural $n$, for a pair of sets of matrices $\{A_{i}\}\in\setA^{n}$ 
and $\{B_{i}\}\in\setB^{n}$ we introduce the value 
\[ 
  \nu_{n}(\{A_{i}\},\{B_{i}\}):= \max_{1\le k\le n}\|A_{k}B_{k}\cdots 
  A_{1}B_{1}\|. 
\] 
Similarly, for each pair of infinite matrix sequences 
$\{A_{i}\}\in\setA^{\bbN}$ and $\{B_{i}\}\in\setB^{\bbN}$, we define the 
value of 
\[ 
  \nu_{\infty}(\{A_{i}\},\{B_{i}\}):= 
  \sup_{n\ge1}\nu_{n}(\{A_{i}\}_{i=1}^{n},\{B_{i}\}_{i=1}^{n}). 
\] 
It is easy to see that the value $\nu_{\infty}(\{A_{i}\},\{B_{i}\})$ also 
admits the representation 
\[ 
  \nu_{\infty}(\{A_{i}\},\{B_{i}\})=\sup_{n\ge1}\|A_{n}B_{n}\cdots 
  A_{1}B_{1}\|. 
\] 
Obviously, the values of $\nu_{n}(\{A_{i}\},\{B_{i}\})$ are finite for any 
integer $n\ge 1$. At the same time, the finiteness of the value of 
$\nu_{\infty}(\{A_{i}\},\{B_{i}\})$ for an arbitrary sequence of matrices 
$\{A_{i}\}\in\setA^{\bbN}$ and $\{B_{i}\}\in\setB^{\bbN}$ follows a'priori 
nowhere. But if the matrix sequence $\{A_{i}\}\in\setA^{\bbN}$ is 
$\setB$-right-bounded, then there exists for it another such sequence 
$\{B_{i}\}\in\setB^{\bbN}$ for which the value 
$\nu_{\infty}(\{A_{i}\},\{B_{i}\})$ is finite. In other words: In this case, 
for any sequence of matrices $\{A_{i}\}\in\setA^{\bbN}$, the inequality 
\[ 
  \inf_{\{B_{i}\}\in\setB^{\bbN}}\nu_{\infty}(\{A_{i}\},\{B_{i}\})<\infty 
\] 
is always valid. 
\begingroup\RenewDocumentCommand\thealtproblem{}{\ref{Q:1}$'$}%
\begin{altproblem}\label{Q:1x} 
  Let $\setA$ and $\setB$ be finite sets of matrices such that every sequence 
  of matrices $\{A_{n}\}\in\setA^{\bbN}$ is $\setB$-right-bounded. In this 
  case, is there a universal constant $C > 0$ such that 
  \[ 
    \inf_{\{B_{i}\}\in\setB^{\bbN}}\nu_{\infty}(\{A_{i}\},\{B_{i}\})\le C 
  \] 
  for any sequence of matrices $\{A_{i}\}\in\setA^{\bbN}$\,? 
\end{altproblem} 
\endgroup 

Question~\ref{Q:1x} can also be rephrased as follows: 

\begingroup\RenewDocumentCommand\thealtproblem{}{\ref{Q:1}$''$}%
\begin{altproblem}\label{Q:1xx} 
  Let $\setA$ and $\setB$ be finite sets of matrices such that every sequence 
  of matrices $\{ A_{n}\}\in\setA^{\bbN}$ is $\setB$-right-bounded. In this 
  case, if the inequality 
  \begin{equation}\label{E:defminf} 
    \mu_{\infty}(\setA,\setB):= 
    \adjustlimits\sup_{\{A_{i}\}\in\setA^{\bbN}}\inf_{\{B_{i}\}\in 
    \setB^{\bbN}}\nu_{\infty}(\{A_{i}\},\{B_{i}\})<\infty 
  \end{equation} 
  holds? 
\end{altproblem} 
\endgroup 

We note that similar issues of pointwise stabilizability of switching systems 
have been considered in~\cite{LJP:HSCC16, JunMas:SIAMJCO17, DJM:SCL20}. 

\section{Main Lemma}\label{S:main} 

We introduce finite analogs of the quantity $\mu_{\infty}(\setA,\setB)$ by 
setting 
\begin{equation}\label{E:defmnen} 
  \mu_{n}(\setA,\setB)=\max_{\{A_{i}\}\in\setA^{n}} 
  \min_{\{B_{i}\}\in\setB^{n}}\nu_{n}(\{A_{i}\},\{B_{i}\}),\qquad 
  n=1,2,\ldots\,. 
\end{equation} 

\begin{lemma}\label{L:1} 
  Inequality~\eqref{E:defminf} is satisfied if and only if there is a constant
  $C>0$ such that
  \begin{equation}\label{E:munbound} 
    \sup_{n\ge1}\mu_{n}(\setA,\setB)\leq C. 
  \end{equation} 
\end{lemma} 

\begin{proof} 
  Let us show that, assuming that~\eqref{E:defminf} is satisfied, there exists
  a constant $C>0$ for which inequality~\eqref{E:munbound} holds. Note first
  that~\eqref{E:defminf} implies the existence of a constant $C'>0$ such that
  \begin{equation}\label{E:defminf-a} 
    \mu_{\infty}(\setA,\setB):= 
    \adjustlimits\sup_{\{A_{i}\}\in\setA^{\bbN}}\inf_{\{B_{i}\}\in 
    \setB^{\bbN}}\nu_{\infty}(\{A_{i}\},\{B_{i}\})\le C'. 
  \end{equation} 
  Let us take an arbitrary natural number $n$ and choose such sets of matrices
  $\{A_{i}\}\in\setA^{n}$ and $\{B_{i}\}\in\setB^{n}$ for which the equality
  \[ 
    \mu_{n}(\setA,\setB)=\nu_{n}(\{A_{i}\},\{B_{i}\}) 
  \] 
  holds (this is possible because the sets $\setA$ and $\setB$ are finite, 
  which implies that the corresponding minima and maxima are reached 
  in~\eqref{E:defmnen}). 

  Furthermore, we extend the set of matrices $\{A_{i}\}\in\setA^{n}$ 
  arbitrarily to the right to the infinite sequence 
  $\{\bar{A}_{i}\}\in\setA^{\bbN}$ and choose an infinite sequence of 
  matrices $\{\bar{B}_{i}\}\in\setB^{\bbN}$ such that the inequality 
  \[ 
    \nu_{\infty}(\{\bar{A}_{i}\},\{\bar{B}_{i}\})\le 
    2\inf_{\{B_{i}\}\in\setB^{\bbN}}\nu_{\infty}(\{\bar{A}_{i}\},\{B_{i}\}) 
  \] 
  is valid (this is possible because of the definition of the value 
  $\nu_{\infty}(\cdot,\cdot)$). If we then set $C=2C'$, we get by 
  definition~\eqref{E:defminf-a}: 
  \[ 
    \nu_{\infty}(\{\bar{A}_{i}\},\{\bar{B}_{i}\})\le 2 
    \mu_{\infty}(\setA,\setB)\le 2C'=C. 
  \] 
  Thus, since the estimate 
  \[ 
    \nu_{n}(\{\bar{A}_{i}\}_{i=1}^{n},\{\bar{B}_{i}\}_{i=1}^{n})\le 
    \nu_{\infty}(\{\bar{A}_{i}\},\{\bar{B}_{i}\}) 
  \] 
  is valid for any natural number $n$, we obtain that 
  \[ 
    \nu_{n}(\{\bar{A}_{i}\}_{i=1}^{n},\{\bar{B}_{i}\}_{i=1}^{n})\le C, 
  \] 
  whereby, by the definition of sets of matrices 
  $\{A_{i}\}=\{\bar{A}_{i}\}_{i=1}^{n}\in\setA^{n}$ and 
  $\{\bar{B}_{i}\}\in\setB^{\bbN}$, we obtain: 
  \[ 
    \mu_{n}(\setA,\setB)=\min_{\{B_{i}\}\in\setB^{n}}\nu_{n}(\{A_{i}\},\{B_{i}\})\le 
    \nu_{n}(\{A_{i}\},\{\bar{B}_{i}\}_{i=1}^{n})= 
    \nu_{n}(\{\bar{A}_{i}\}_{i=1}^{n},\{\bar{B}_{i}\}_{i=1}^{n})\le C. 
  \] 
  In one direction the assertion of the lemma is proved. 

  Let us prove the assertion of the lemma in the opposite direction: Suppose 
  that inequalities~\eqref{E:munbound} hold, and prove that in this case the 
  inequality 
  \begin{equation}\label{E:defminf-b} 
    \mu_{\infty}(\setA,\setB):= 
    \adjustlimits\sup_{\{A_{i}\}\in\setA^{\bbN}}\inf_{\{B_{i}\}\in 
    \setB^{\bbN}}\nu_{\infty}(\{A_{i}\},\{B_{i}\})\le C 
  \end{equation} 
  also holds. To do this, choose an arbitrary sequence of matrices 
  $\{A_{i}\}\in\setA^{\bbN}$, and then for each $n=1,2,\ldots$ choose a set 
  of matrices $\{B_{i}^{(n)}\}\in\setB^{n}$ that satisfy the inequalities 
  \[ 
    \nu_{n}(\{A_{i}\}_{i=1}^{n},\{B_{i}^{(n)}\}_{i=1}^{n})\le 
    \mu_{n}(\setA,\setB)\le C. 
  \] 
  Then for each $k\ge n$, due to the definition of the quantities 
  $\nu_{n}(\cdot,\cdot)$, the inequalities 
  \begin{equation}\label{E:nuikn} 
    \max_{1\le j\le n}\|A_{j}B_{j}^{(k)}\cdots 
    A_{1}B_{1}^{(k)}\|=\nu_{n}(\{A_{i}\}_{i=1}^{n},\{B_{i}^{(k)}\}_{i=1}^{n})\le 
    \mu_{n}(\setA,\setB)\le C,\qquad n=1,2,\ldots\,, 
  \end{equation} 
  hold. Further, since for each natural number $i$ the elements of the matrix 
  sequence $\{B_{i}^{(k)}\}_{k=1}^{\infty}$ are taken from a finite set, we 
  can assume without loss of generality that there exist limits 
  \[ 
    \lim_{k\to\infty}B_{i}^{(k)}=\bar{B}_{i}\in\setB,\qquad i=1,2,\ldots\,. 
  \] 

  Passing to the limit in~\eqref{E:nuikn} as $k\to\infty$, we obtain 
  \[ 
    \max_{1\le j\le n}\|A_{j}\bar{B}_{j}\cdots 
    A_{1}\bar{B}_{1}\|=\nu_{n}(\{A_{i}\}_{i=1}^{n},\{\bar{B}_{i}\}_{i=1}^{n})\le 
    \mu_{n}(\setA,\setB)\le C,\qquad n=1,2,\ldots\,, 
  \] 
  or, equivalently, 
  \[ 
    \sup_{n\ge 1}\|A_{n}\bar{B}_{n}\cdots A_{1}\bar{B}_{1}\|\le C. 
  \] 
  Therefore, for the chosen sequence of matrices $\{A_{i}\}\in\setA^{\bbN}$ 
  and the constructed sequence of matrices $\{\bar{B}_{i}\}\in\setB^{\bbN}$ 
  the relations 
  \[ 
    \nu_{\infty}(\{A_{i}\},\{\bar{B}_{i}\})= 
    \sup_{n\ge1}\nu_{n}(\{A_{i}\}_{i=1}^{n},\{\bar{B}_{i}\}_{i=1}^{n})=\sup_{n\ge 
    1}\|A_{n}\bar{B}_{n}\cdots A_{1}\bar{B}_{1}\|\le C 
  \] 
  hold. Since the sequence of matrices $\{A_{i}\}\in\setA^{\bbN}$ is 
  arbitrary, this implies inequality~\eqref{E:defminf-b}. The lemma is proved.
\end{proof} 

\section{Right- and Left-Bounded Matrix Products} 

When examining matrix products with alternating factors from two sets of 
matrices, instead of considering products of the form 
\[ 
  A_{n}B_{n}\cdots A_{1}B_{1},\quad\text{where}\quad 
  A_{i}\in\setA,~B_{i}\in\setB, 
\] 
where the factors $B_{i}$ lie `to the right' of the corresponding factors 
$A_{i}$, one can also consider matrix products of the form 
\[ 
  B_{n}A_{n}\cdots B_{1}A_{1},\quad\text{where}\quad 
  A_{i}\in\setA,~B_{i}\in\setB, 
\] 
where the factors $B_{i}$ lie `to the left' of the corresponding factors 
$A_{i}$. It is more or less clear that the difference between these two cases 
should not affect the question of convergence or boundedness of the 
respective products. However, for the sake of completeness, we present more 
formalized statements and arguments below. 

We say that a sequence of matrices $\{A_{n}\}\in\setA^{\bbN}$ is 
\emph{$\setB$-left-bounded}, if there is a sequence of matrices 
$\{B_{n}\}\in\setB^{\bbN}$ for which the sequence of norms 
$\{\|B_{n}A_{n}\cdots B_{1}A_{1}\|\}_{n=1}^{\infty}$ is bounded, i.e., there 
exists a constant $\tilde{C}=\tilde{C}(\{A_{n}\})$ such that 
\[ 
  \exists~\{B_{n}\}\in\setB^{\bbN}:\quad \|B_{n}A_{n}\cdots B_{1}A_{1}\|\le 
  \tilde{C}, \quad n=1,2,\ldots\,. 
\] 

\begin{lemma}\label{L:ABeqBA} 
  Let the sets of matrices $\setA$ and $\setB$ be bounded. Then if every 
  sequence of matrices $\{A_{n}\}\in\setA^{\bbN}$ is $\setB$-left-bounded, 
  then every sequence of matrices $\{A_{n}\}\in\setA^{\bbN}$ is also 
  $\setB$-right bounded. And if every sequence of matrices 
  $\{A_{n}\}\in\setA^{\bbN}$ is $\setB$-right-bounded, then every sequence of 
  matrices $\{A_{n}\}\in\setA^{\bbN}$ is also $\setB$-left-bounded. 
\end{lemma} 

\begin{proof} 
  Since by the condition of the lemma the sets of matrices $\setA$ and $\setB$
  are bounded, there are constants $a$ and $b$ such that
  \[ 
    \sup\{\|A\|: A\in\setA\}\le a,\qquad \sup\{\|B\|: B\in\setB\}\le b. 
  \] 
  Let us first consider the case where any sequence of matrices with values in
  $\setA$ is $\setB$-right-bounded. We take an arbitrary sequence of matrices
  $\{A_{n}\}\in\setA^{\bbN}$ and prove that in this case it is
  $\setB$-left-bounded. 

  By assumption, the `shifted' sequence of matrices 
  $\{\tilde{A}_{n}\}\in\setA^{\bbN}$ defined by the equalities 
  \begin{equation}\label{E:deftildeA} 
    \tilde{A}_{n}=A_{n+1}, \quad n=1,2,\ldots\,, 
  \end{equation} 
  is $\setB$-right-bounded. Then due to~\eqref{E:defABbounded} there is a 
  sequence of matrices $\{B_{n}\}\in\setB^{\bbN}$ and a constant 
  $C=C(\{\tilde{A}_{n}\})$ such that 
  \[ 
    \|\tilde{A}_{n}B_{n}\cdots \tilde{A}_{1}B_{1}\|\le C, \quad 
    n=1,2,\ldots\,. 
  \] 
  From this and from the definition~\eqref{E:deftildeA} 
  \begin{multline*} 
    \|B_{n}A_{n}\cdots B_{1}A_{1}\| = \|B_{n}\left(A_{n}B_{n-1}\cdots 
    A_{2}B_{1}\right)A_{1}\|= \|B_{n}\left(\tilde{A}_{n-1}B_{n-1}\cdots 
    \tilde{A}_{1}B_{1}\right)A_{1}\|\le\\ \le \|B_{n}\| 
    \|\tilde{A}_{n-1}B_{n-1}\cdots \tilde{A}_{1}B_{1}\|
    \|A_{1}\|\le bCa, \quad n=2,3,\ldots\,.
  \end{multline*} 
  This proves the $\setB$-left-boundedness of the sequence 
  $\{A_{n}\}\in\setA^{\bbN}$. 

  The $\setB$-right-boundedness of any sequence of matrices 
  $\{A_{n}\}\in\setA^{\bbN}$ is proved similarly, provided that any sequence 
  of matrices with values in $\setA$ is $\setB$-left-bounded. 
\end{proof} 

It follows from Lemma~\ref{L:ABeqBA} that when considering 
Question~\ref{Q:1}, instead of $\setB$-right-boundedness one can assume 
$\setB$-left-boundedness of the corresponding matrix products. 

\section{Nonsingular Matrices}\label{S:invertible} 

In this section we will assume\footnote{Recall that the matrix sets $\setA$ 
  and $\setB$ do not necessarily consist of square matrices, so it is natural
to define the nonsingularity of these sets by the products $AB$ or $BA$.} 
that 
\begin{equation}\label{E:Cbound} 
  \det(AB)\neq 0,\qquad \forall~ A\in\setA,~B\in\setB. 
\end{equation} 

The matrix $AB$ has dimension $N\times N$. Moreover, its rank is not greater 
than the minimum of the ranks of the matrices $A$ and $B$, i.e., not greater 
than $\min\{M,N\}$. Therefore, the relation~\eqref{E:Cbound} in which the 
matrix $B$ is applied to the matrix $A$ on the right-hand side can hold only 
if $N\le M$. 

In the case where $N > M$, the relation 
\[ 
  \det(BA)\neq 0,\qquad \forall~ A\in\setA,~B\in\setB, 
\] 
may be valid in which the matrix $B$ is applied to the matrix $A$ on the 
left-hand side. Accordingly, in this case, further argumentation in this 
section should be with $\setB$-left-bounded matrix products. 

Let us show that under condition~\eqref{E:Cbound} the answer to 
Question~\ref{Q:1} is positive: 
\begin{theorem}\label{T1} 
  Let $\setA$ and $\setB$ be finite sets of matrices satisfying 
  condition~\eqref{E:Cbound}, and let every sequence of matrices 
  $\{A_{n}\}\in\setA^{\bbN}$ be $\setB$-right-bounded. Then there is a 
  constant $C>0$ such that for any sequence of matrices 
  $\{A_{n}\}\in\setA^{\bbN}$ inequalities~\eqref{E:defABbounded} hold. 
\end{theorem} 

\begin{proof} 
  Let us prove the required assertion by a contradiction: Show that the 
  absence of the required constant $C$ implies the existence of such a 
  sequence of matrices $\{A_{n}\}\in\setA^{\bbN}$, for which the sequence of 
  norms $\{\|A_{n}B_{n}\cdots A_{1}B_{1}\|\}_{n=1}^{\infty}$ is unbounded for 
  any choice of the sequence of matrices $\{B_{n}\}\in\setB^{\bbN}$. 

  So suppose that the required constant $C$ does not exist, then 
  inequalities~\eqref{E:munbound} are true for no $C$ and therefore 
  \begin{equation}\label{E:mnunbound} 
    \limsup_{n\to\infty}\mu_{n}(\setA,\setB) =\infty. 
  \end{equation} 

  Let us take an arbitrary positive number $\varkappa_{1}$, for example 
  $\varkappa_{1}=1$. Then, due to~\eqref{E:mnunbound}, there is an integer 
  $n_{1}\ge1$ such that 
  \begin{equation}\label{E:k1} 
    \mu_{n_{1}}(\setA,\setB)\ge\varkappa_{1}=1. 
  \end{equation} 
  Let us denote by $\bar{A}_{1},\ldots,\bar{A}_{n_{1}}\in\setA$ the set of 
  matrices for which the right side in~\eqref{E:defmnen} for $n=n_{1}\ge1$ 
  reaches the maximum over $\{A_{i}\}\in\setA^{n_{1}}$. Then, 
  given~\eqref{E:k1} and the definition~\eqref{E:defmnen} of the quantities 
  $\mu_{n}(\setA,\setB)$, we have 
  \[ 
    \min_{\{B_{i}\}\in\setB^{n_{1}}} 
    \nu_{n_{1}}(\{\bar{A}_{i}\}_{i=1}^{n_{1}},\{B_{i}\})= 
    \mu_{n_{1}}(\setA,\setB)\ge\varkappa_{1}=1 
  \] 
  and thus by the definition of $\nu_{n}(\cdot,\cdot)$ for any sequence of 
  matrices $\{B_{i}\}\in\setB^{n_{1}}$ the following relations hold: 
  \begin{equation}\label{E:barAB} 
    \max_{1\le k\le n_{1}}\|\bar{A}_{k}B_{k}\cdots 
    \bar{A}_{1}B_{1}\|=\nu_{n_{1}}(\{\bar{A}_{i}\}_{i=1}^{n_{1}},\{B_{i}\})= 
    \mu_{n_{1}}(\setA,\setB)\ge\varkappa_{1}=1. 
  \end{equation} 

  Now note that due to the assumption that the sets $\setA$ and $\setB$ are 
  finite, the set of matrices $\bar{A}_{n_{1}}B_{n_{1}}\cdots 
  \bar{A}_{1}B_{1}$ in~\eqref{E:barAB} is also finite and all these matrices 
  are invertible due to the assumption~\eqref{E:Cbound}. Then one can specify 
  a constant $\eta_{1} > 0$ such that 
  \begin{equation}\label{E:barABinv} 
    \|(\bar{A}_{n_{1}}B_{n_{1}}\cdots 
    \bar{A}_{1}B_{1})^{-1}\|\le\eta_{1},\qquad\forall~\{B_{i}\}\in\setB^{n_{1}}. 
  \end{equation} 

  Let us now choose an arbitrary number $\varkappa_{2}\ge 2\eta_{1}$. Then, 
  again due to~\eqref{E:mnunbound}, there is an integer $n_{2}\ge 1$ such that
  \begin{equation}\label{E:k2} 
    \mu_{n_{2}}(\setA,\setB)\ge\varkappa_{2}. 
  \end{equation} 
  Let $\bar{A}_{n_{1}+1},\ldots,\bar{A}_{n_{1}+n_{2}}$ denote the set of 
  matrices from $\setA$ for which the right-hand side in~\eqref{E:defmnen} 
  for $n=n_{2}\ge1$ reaches its maximum over $\{A_{i}\}\in\setA^{n_{2}}$. 
  Then, due to~\eqref{E:k2} 
  \[ 
    \min_{\{B_{i}\}\in\setB^{n_{2}}} 
    \nu_{n_{2}}(\{\bar{A}_{i}\}_{i=n_{1}+1}^{n_{2}},\{B_{i}\})= 
    \mu_{n_{2}}(\setA,\setB)\ge\varkappa_{2} 
  \] 
  and thus by the definition of $\nu_{n}(\cdot,\cdot)$ the relations 
  \begin{equation}\label{E:barAB2} 
    \max_{1\le k\le n_{2}}\|\bar{A}_{n_{1}+k}B_{n_{1}+k}\cdots 
    \bar{A}_{n_{1}+1}B_{n_{1}+1}\|= 
    \nu_{n_{2}}(\{\bar{A}_{i}\}_{i=n_{1}+1}^{n_{1}+n_{2}},\{B_{i}\})= 
    \mu_{n_{2}}(\setA,\setB)\ge\varkappa_{2} 
  \end{equation} 
  hold for any sequence of matrices $\{B_{i}\}\in\setB^{n_{2}}$. 

  Now consider for $k=1,2,\ldots,n_{2}$ the product of the matrices 
  \[ 
    W_{k}=\bar{A}_{n_{1}+k}B_{n_{1}+k}\cdots \bar{A}_{n_{1}+1}B_{n_{1}+1}\cdot
    \bar{A}_{n_{1}}B_{n_{1}}\cdots \bar{A}_{1}B_{1}
  \] 
  and estimate its norm from below. To this end, we note that for any 
  matrices $P$ and $Q$ of which the matrix $Q$ is invertible, the relations 
  $\|P\|=\| PQ \cdot Q^{-1}\|\le \|PQ\|\cdot\|Q^{-1}\|$ hold, from which 
  \[ 
    \|PQ\|\ge \|P\|\cdot \|Q^{-1}\|^{-1}. 
  \] 
  So for the matrix $W_{k}=P_{k}Q$, where 
  \[ 
    P_{k}=\bar{A}_{n_{1}+k}B_{n_{1}+k}\cdots 
    \bar{A}_{n_{1}+1}B_{n_{1}+1},\quad
    Q=\bar{A}_{n_{1}}B_{n_{1}}\cdots \bar{A}_{1}B_{1},
  \] 
  we obtain for each $k=1,2,\ldots,n_{2}$ the inequalities 
  \begin{multline}\label{E:prodn2} 
    \|\bar{A}_{n_{1}+k}B_{n_{1}+k}\cdots \bar{A}_{n_{1}+1}B_{n_{1}+1}\cdot 
    \bar{A}_{n_{1}}B_{n_{1}}\cdots \bar{A}_{1}B_{1}\|=\|W_{k}\|=\|P_{k}Q\|\ge 
    \|P_{k}\|\cdot \|Q^{-1}\|^{-1}\ge\\ \ge 
    \|\bar{A}_{n_{1}+k}B_{n_{1}+k}\cdots \bar{A}_{n_{1}+1}B_{n_{1}+1}\|\cdot 
    \|(\bar{A}_{n_{1}}B_{n_{1}}\cdots \bar{A}_{1}B_{1})^{-1}\|^{-1}. 
  \end{multline} 
  Because of~\eqref{E:barAB2}, here the first factor on the right-hand side 
  satisfies the estimate 
  \[ 
    \max_{1\le k\le n_{2}}\|\bar{A}_{n_{1}+k}B_{n_{1}+k}\cdots 
    \bar{A}_{n_{1}+1}B_{n_{1}+1}\|\ge\varkappa_{2}\ge2\eta_{1},\qquad 
    \forall~\{B_{n_{1}+i}\}\in\setB^{n_{2}}, 
  \] 
  and for the second factor, due to~\eqref{E:barABinv}, the estimate 
  \[ 
    \|(\bar{A}_{n_{1}}B_{n_{1}}\cdots 
    \bar{A}_{1}B_{1})^{-1}\|^{-1}\ge\eta_{1}^{-1},\qquad 
    \forall~\{B_{i}\}\in\setB^{n_{1}}, 
  \] 
  takes place. Consequently, due to~\eqref{E:prodn2} 
  \[ 
    \max_{1\le k\le n_{2}}\|\bar{A}_{n_{1}+n_{2}}B_{n_{1}+n_{2}}\cdots 
    \bar{A}_{n_{1}+1}B_{n_{1}+1}\cdot \bar{A}_{n_{1}}B_{n_{1}}\cdots 
    \bar{A}_{1}B_{1}\|\ge 2,\qquad \forall~\{B_{i}\}\in\setB^{n_{1}+n_{2}}, 
  \] 
  and therefore, even more, 
  \[ 
    \max_{1\le k\le n_{1}+n_{2}}\|\bar{A}_{k}B_{k}\cdots 
    \bar{A}_{1}B_{1}\|\ge\max_{1\le k\le 
    n_{2}}\|\bar{A}_{n_{1}+k}B_{n_{1}+k}\cdots 
    \bar{A}_{n_{1}+1}B_{n_{1}+1}\cdot \bar{A}_{n_{1}}B_{n_{1}}\cdots 
    \bar{A}_{1}B_{1}\|\ge 2 
  \] 
  for any sequence of matrices $\{B_{i}\}\in\setB^{n_{1}+n_{2}}$. 

  Similarly, we define a constant $\eta_{2} > 0$ to satisfy the inequality 
  \[ 
    \|(\bar{A}_{n_{1}+n_{2}}B_{n_{1}+n_{2}}\cdots 
    \bar{A}_{1}B_{1})^{-1}\|\le\eta_{2},\qquad\forall~\{B_{i}\}\in\setB^{n_{1}+n_{2}}, 
  \] 
  and choose an arbitrary number $\varkappa_{3}\ge 3\eta_{2}$. Then, as in 
  the previous case, one can specify such an integer $n_{3}\ge 1$ and a set 
  of matrices $\bar{A}_{n_{1}+n_{2}+1},\ldots,\bar{A}_{n_{1}+n_{2}+n_{3}}$ 
  for which 
  \[ 
    \max_{1\le k\le n_{1}+n_{2}+n_{3}}\|\bar{A}_{k}B_{k}\cdots 
    \bar{A}_{1}B_{1}\|\ge 3,\qquad 
    \forall~\{B_{i}\}\in\setB^{n_{1}+n_{2}+n_{3}}. 
  \] 

  Acting similarly, for each $m=1,2,\ldots$ we can specify an integer 
  $n_{m}\ge 1$ and have a set of matrices 
  \[ 
    \bar{A}_{n_{1}+\cdots+n_{m-1}+1},\ldots,\bar{A}_{n_{1}+\cdots+n_{m-1}+n_{m}} 
  \] 
  such that 
  \begin{equation}\label{E:prodnk} 
    \max_{1\le k\le n_{1}+\cdots+n_{m}}\|\bar{A}_{k}B_{k}\cdots 
    \bar{A}_{1}B_{1}\|\ge m,\qquad 
    \forall~\{B_{i}\}\in\setB^{n_{1}+\cdots+n_{m}}. 
  \end{equation} 

  The resulting estimate contradicts the assumption from Question~\ref{Q:1} 
  that for every sequence of matrices $\{A_{n}\}\in\setA^{\bbN}$ there is a 
  sequence of matrices $\{B_{n}\}\in\setB^{\bbN}$ for which the sequence of 
  norms $\{\|A_{n}B_{n}\cdots A_{1}B_{1}\|,~n=1,2,\ldots\}$ is bounded. In 
  fact, due to inequalities~\eqref{E:prodnk} for the constructed sequence of 
  matrices $\{\bar{A}_{i}\}\in\setA^{\bbN}$ and \textbf{any} sequence of 
  matrices $\{B_{i}\}\in\setB^{\bbN}$, the sequence of norms 
  \[ 
    \{\|\bar{A}_{n}B_{n}\cdots \bar{A}_{1}B_{1}\|,~n=1,2,\ldots\} 
  \] 
  is not bounded. Theorem~\ref{T1} is proved. 
\end{proof} 

\section{Nonnegative Row Nondegenerate Matrices}\label{S:positive2} 

The purpose of this section is to relax the restrictions imposed on the 
matrix sets $\setA$ and $\setB$ in Section~\ref{S:invertible}. 

As in Section~\ref{S:invertible}, the sets of matrices $\setA$ and $\setB$ 
are assumed to be finite, but instead of the nonsingularity~\eqref{E:Cbound} 
of the products of matrices $AB$, where $A\ in\setA$, $B\in\setB$, the 
corresponding matrices are assumed to be nonnegative: 
\begin{equation}\label{E:AB-nonneg} 
  a_{ij}\ge0,\quad b_{ij}\ge0,\qquad \forall~ 
  A=(a_{ij})\in\setA,~B=(b_{ij})\in\setB, 
\end{equation} 
and it is also assumed that each matrix of $\setA$ and $\setB$ is \emph{row 
nondegenerate}, i.e., cannot have rows consisting entirely of zero elements: 
\begin{equation}\label{E:AB-nonzerostrings} 
  \min_{i}\sum_{j}a_{ij}>0,\quad\min_{i}\sum_{j}b_{ij}>0,\qquad \forall~ 
  A=(a_{ij})\in\setA,~B=(b_{ij})\in\setB. 
\end{equation} 
Let us show that under these conditions the answer to Question~\ref{Q:1} is 
positive: 

\begin{theorem}\label{T2} 
  Let $\setA$ and $\setB$ be finite sets of matrices 
  satisfying~\eqref{E:AB-nonneg} and~\eqref{E:AB-nonzerostrings}, and let any 
  sequence of matrices $\{A_{n}\}\in\setA^{\bbN}$ be $\setB$-right-bounded. 
  Then there is a constant $C > 0$ such that for any sequence of matrices 
  $\{A_{n}\}\in\setA^{\bbN}$ inequalities~\eqref{E:defABbounded} hold. 
\end{theorem} 

Before proceeding with the proof of Theorem~\ref{T2}, we make a few remarks 
and note a number of auxiliary facts. 

Note that conditions~\eqref{E:AB-nonneg} and~\eqref{E:AB-nonzerostrings} are 
satisfied by the sets of matrices $\setA$ and $\setB$, each of which has 
strictly positive elements, i.e., 
\begin{equation}\label{E:AB-positive} 
  \min_{i,j}a_{ij} > 0,\quad\min_{i,j}b_{ji} > 0,\qquad \forall~ 
  A=(a_{ij})\in\setA,~B=(b_{ij})\in\setB. 
\end{equation} 

Note also that under condition~\eqref{E:AB-nonzerostrings}, since the sets 
$\setA$ and $\setB$ are finite, there is $\gamma>0$ such that 
\[ 
  \min_{i}\sum_{j}a_{ij}\ge\gamma,\quad\min_{i}\sum_{j}b_{ij}\ge\gamma,\qquad 
  \forall~ A=(a_{ij})\in\setA,~B=(b_{ij})\in\setB. 
\] 
If we introduce a finite set of matrices 
\[ 
  \setF=\{AB\colon A\in\setA,~B\in\setB\}\subset\setM(N,N), 
\] 
then we can assume without loss of generality that under 
conditions~\eqref{E:AB-nonneg} and~\eqref{E:AB-nonzerostrings} the following 
inequality holds: 
\begin{equation}\label{E:nonzerostrings} 
  \min_{i}\sum_{j}f_{ij}\ge\gamma,\qquad \forall~ F=(f_{ij})\in\setF. 
\end{equation} 

Finally, note that if condition~\eqref{E:AB-nonneg} is satisfied, 
condition~\eqref{E:nonzerostrings} follows from, but is not equivalent to, 
condition~\eqref{E:AB-nonzerostrings}. 

By $e$ we denote the vector in $\mathbb{R}^{N}$ (or $\mathbb{R}^{M}$) 
consisting of the unit elements\footnote{The fact that the same notation 
applies to elements from different spaces will hopefully not cause any 
confusion!}: 
\[ 
  e=\{1,1,\ldots,1\}. 
\] 
Then the following equivalent description of 
conditions~\eqref{E:AB-nonzerostrings} and~\eqref{E:nonzerostrings} is true. 
\begin{lemma}\label{L:e-boundedness} 
  Condition~\eqref{E:AB-nonzerostrings} is true if and only if there exists 
  $\gamma > 0$ such that 
  \[ 
    Ae\ge\gamma e,\quad Be\ge\gamma e,\qquad \forall~ A\in\setA,~B\in\setB, 
  \] 
  where the inequalities between vectors are understood coordinate-wise. 

  If condition~\eqref{E:AB-nonneg} holds, then 
  condition~\eqref{E:nonzerostrings} is true if and only if there is 
  $\gamma>0$ such that 
  \[ 
    Fe\ge\gamma e,\qquad \forall~ F\in\setF, 
  \] 
  where the inequalities between vectors are understood coordinate-wise. 
\end{lemma} 

\begin{proof} 
  Evident. 
\end{proof} 

Note a general fact related to nonnegative matrices. At the end of this 
section, let $\|\cdot\|$ be the so-called max-norm in $\mathbb{R}^{N}$, i.e., 
the norm given by the equality 
\[ 
  \|x\|=\max_{i} |x_{i}|,\qquad x=\{x_{1},x_{2},\ldots,x_{N}\}. 
\] 
\begin{lemma}\label{L:maxnorm} 
  If $A\in\setM(N,N)$ is a matrix with nonnegative elements, then 
  $\|A\|=\|Ae\|$, where $e=\{1,1,\ldots,1\}$. 
\end{lemma} 
\begin{proof} 
  Evident. 
\end{proof} 

According to Lemma~\ref{L:maxnorm}, for any matrix $A$ with nonnegative 
entries, the equality $\|A\|=\|Ae\|$ holds. Therefore, in the following, when 
estimating the norm of the matrix products under consideration, it suffices 
to estimate the norms of the matrices on the element $e$. 

Now we can proceed with the proof of Theorem~\ref{T2}. 

\begin{proof}[Proof of Theorem~\ref{T2}] 
  As in the case of Theorem~\ref{T1}, we conduct the proof by contradiction: 
  We will show that the absence of the required constant $C$ implies the 
  existence of such a sequence of matrices $\{A_{n}\}\in\setA^{\bbN}$, for 
  which the sequence of norms $\{\|A_{n}B_{n}\cdots 
  A_{1}B_{1}\|\}_{n=1}^{\infty}$ is unbounded for any choice of the sequence 
  of matrices $\{B_{n}\}\in\setB^{\bbN}$. 

  Thus, suppose that the required constant $C$ does not exist, then 
  inequalities~\eqref{E:Cbound} are not true for any $C$, i.e., 
  \begin{equation}\label{E:mnunbound3} 
    \limsup_{n\to\infty}\mu_{n}(\setA,\setB) =\infty. 
  \end{equation} 

  Let us take an arbitrary positive number $\varkappa_{1}$, for example 
  $\varkappa_{1}=1$. Then, due to~\eqref{E:mnunbound3}, there is an integer 
  $n_{1}\ge1$ such that 
  \begin{equation}\label{E:k123} 
    \mu_{n_{1}}(\setA,\setB)\ge\varkappa_{1}=1. 
  \end{equation} 
  Let $\bar{A}_{1},\ldots,\bar{A}_{n_{1}}$ denote the set of matrices from 
  $\setA$ for which the right-hand side in~\eqref{E:defmnen} for $n=n_{1}\ge1$
  reaches the maximum over $\{A_{i}\}\in\setA^{n_{1}}$. Then, due
  to~\eqref{E:k123} and the definition~\eqref{E:defmnen} of the quantities
  $\mu_{n}(\setA,\setB)$,
  \[ 
    \min_{\{B_{i}\}\in\setB^{n_{1}}} 
    \nu_{n_{1}}(\{\bar{A}_{i}\}_{i=1}^{n_{1}},\{B_{i}\})= 
    \mu_{n_{1}}(\setA,\setB)\ge\varkappa_{1}=1 
  \] 
  and thus by the definition of $\nu_{n}(\cdot,\cdot)$ the relations 
  \[ 
    \max_{1\le k\le n_{1}}\|\bar{A}_{k}B_{k}\cdots 
    \bar{A}_{1}B_{1}\|=\nu_{n_{1}}(\{\bar{A}_{i}\}_{i=1}^{n_{1}},\{B_{i}\})= 
    \mu_{n_{1}}(\setA,\setB)\ge\varkappa_{1}=1 
  \] 
  hold for any sequence of matrices $\{B_{i}\}\in\setB^{n_{1}}$. 

  Now note that because of the assumptions~\eqref{E:AB-nonzerostrings} 
  or~\eqref{E:nonzerostrings} for any sequences of matrices 
  $\{B_{i}\}\in\setB^{n_{1}}$ each vector $\bar{A}_{n_{1}}B_{n_{1}}\cdots 
  \bar{A}_{1}B_{1}e$ does not vanish and, moreover, is coordinate-bounded from
  below by a positive factor of the vector $e$. Since the set of matrices
  $\setB$ is finite, it follows that there is a constant $\omega_{1} > 0$ for
  which 
  \begin{equation}\label{E:mainlowbound3} 
    \bar{A}_{n_{1}}B_{n_{1}}\cdots \bar{A}_{1}B_{1}e\ge\omega_{1}e,\qquad 
    \forall~ \{B_{i}\}\in\setB^{n_{1}}. 
  \end{equation} 

  Let us now choose an arbitrary number $\varkappa_{2}\ge 2\omega_{1}^{-1}$. 
  Then, again due to~\eqref{E:mnunbound3}, there is an integer $n_{2}\ge 1$ 
  such that 
  \begin{equation}\label{E:k2pos3} 
    \mu_{n_{2}}(\setA,\setB)\ge\varkappa_{2}. 
  \end{equation} 
  Let $\bar{A}_{n_{1}+1},\ldots,\bar{A}_{n_{1}+n_{2}}$ denote the set of 
  matrices from $\setA$ for which the right-hand side in~\eqref{E:defmnen} 
  for $n=n_{2}\ge1$ reaches its maximum over $\{A_{i}\}\in\setA^{n_{2}}$. 
  Then, due to~\eqref{E:k2pos3} 
  \[ 
    \min_{\{B_{i}\}\in\setB^{n_{2}}} 
    \nu_{n_{2}}(\{\bar{A}_{i}\}_{i=n_{1}+1}^{n_{2}},\{B_{i}\})= 
    \mu_{n_{2}}(\setA,\setB)\ge\varkappa_{2} 
  \] 
  and therefore by the definition of $\nu_{n}(\cdot,\cdot)$ for any sequence 
  of matrices $\{B_{i}\}\in\setB^{n_{2}}$ the following relations hold: 
  \[ 
    \max_{1\le k\le n_{2}}\|\bar{A}_{n_{1}+k}B_{n_{1}+k}\cdots 
    \bar{A}_{n_{1}+1}B_{n_{1}+1}\|= 
    \nu_{n_{2}}(\{\bar{A}_{i}\}_{i=n_{1}+1}^{n_{1}+n_{2}},\{B_{i}\})= 
    \mu_{n_{2}}(\setA,\setB)\ge\varkappa_{2}, 
  \] 
  which by Lemma~\ref{L:maxnorm} are equivalent to the inequalities 
  \begin{equation}\label{E:barAB2pos3} 
    \max_{1\le k\le n_{2}}\|\bar{A}_{n_{1}+k}B_{n_{1}+k}\cdots 
    \bar{A}_{n_{1}+1}B_{n_{1}+1}e\|\ge\varkappa_{2}. 
  \end{equation} 

  Now consider for $k=1,2,\ldots,n_{2}$ the product of the matrices 
  \[ 
    W_{k}=\bar{A}_{n_{1}+k}B_{n_{1}+k}\cdots \bar{A}_{n_{1}+1}B_{n_{1}+1}\cdot
    \bar{A}_{n_{1}}B_{n_{1}}\cdots \bar{A}_{1}B_{1}
  \] 
  and estimate its norm $\|W_{k}\|$ from below. For this, due to 
  Lemma~\ref{L:maxnorm}, it suffices to estimate the norm of the vector 
  $\|W_{k}e\|=\|W_{k}\|$ from below. 

  We set 
  \[ 
    P_{k}=\bar{A}_{n_{1}+k}B_{n_{1}+k}\cdots 
    \bar{A}_{n_{1}+1}B_{n_{1}+1},\quad Q=\bar{A}_{n_{1}}B_{n_{1}}\cdots 
    \bar{A}_{1}B_{1}. 
  \] 
  Now note that in the representation $W_{k}=P_{k}Qe$, the lower bound holds 
  for the vector $Qe$ due to~\eqref{E:mainlowbound3}: 
  \[ 
    Qe\ge \omega_{1} e, 
  \] 
  which, since the matrices $P_{k}$ are positive, for each 
  $k=1,2,\ldots,n_{2}$ the relations 
  \[ 
    \bar{A}_{n_{1}+k}B_{n_{1}+k}\cdots \bar{A}_{n_{1}+1}B_{n_{1}+1}\cdot 
    \bar{A}_{n_{1}}B_{n_{1}}\cdots \bar{A}_{1}B_{1}e=W_{k}e=P_{k}Qe\ge 
    P_{k}\left(\omega_{1}e\right) = \omega_{1} P_{k}e. 
  \] 
  From this and from Lemma~\ref{L:maxnorm} 
  \[ 
    \|\bar{A}_{n_{1}+k}B_{n_{1}+k}\cdots \bar{A}_{n_{1}+1}B_{n_{1}+1}\cdot 
    \bar{A}_{n_{1}}B_{n_{1}}\cdots \bar{A}_{1}B_{1}\|\ge \omega_{1} 
    \|P_{k}e\|, 
  \] 
  and then, due to inequalities~\eqref{E:barAB2pos3}, 
  \[ 
    \max_{1\le k\le n_{2}}\|\bar{A}_{n_{1}+k}B_{n_{1}+k}\cdots 
    \bar{A}_{n_{1}+1}B_{n_{1}+1}\cdot \bar{A}_{n_{1}}B_{n_{1}}\cdots 
    \bar{A}_{1}B_{1}\| \ge \omega_{1} \varkappa_{2}\ge 
    2,\qquad\forall~\{B_{i}\}\in\setB^{n_{1}+n_{2}}, 
  \] 
  and hence, more, for any sequence of matrices 
  $\{B_{i}\}\in\setB^{n_{1}+n_{2}}$ we have the estimates 
  \[ 
    \max_{1\le k\le n_{1}+n_{2}}\|\bar{A}_{k}B_{k}\cdots 
    \bar{A}_{1}B_{1}\|\ge\max_{1\le k\le 
    n_{2}}\|\bar{A}_{n_{1}+k}B_{n_{1}+k}\cdots 
    \bar{A}_{n_{1}+1}B_{n_{1}+1}\cdot
    \bar{A}_{n_{1}}B_{n_{1}}\cdots \bar{A}_{1}B_{1}\|\ge 2.
  \] 

  Similarly, we define the constant $\omega_{2}>0$ to satisfy the inequality 
  \[ 
    \bar{A}_{n_{1}+n_{2}}B_{n_{1}+n_{2}}\cdots 
    \bar{A}_{1}B_{1}e\ge\omega_{2}e,\qquad\forall~\{B_{i}\}\in\setB^{n_{1}+n_{2}}, 
  \] 
  and now choose an arbitrary number $\varkappa_{3}\ge 3\omega_{2}^{-1}$. 

  Then, as in the previous case, one can specify such an integer $n_{3}\ge 1$ 
  and a set of matrices 
  $\bar{A}_{n_{1}+n_{2}+1},\ldots,\bar{A}_{n_{1}+n_{2}+n_{3}}$, for which the 
  inequality 
  \[ 
    \max_{1\le k\le n_{3}}\|\bar{A}_{n_{1}+n_{2}+k}B_{n_{1}+n_{2}+k}\cdots 
    \bar{A}_{n_{1}+n_{2}+1}B_{n_{1}+n_{2}+1}e\|\ge\varkappa_{2},\qquad\forall~ 
    \{B_{n_{1}+n_{2}+i}\}\in\setB^{n_{3}}, 
  \] 
  holds, and therefore also the inequality 
  \[ 
    \max_{1\le k\le n_{1}+n_{2}+n_{3}}\|\bar{A}_{k}B_{k}\cdots 
    \bar{A}_{1}B_{1}\|\ge 3,\qquad 
    \forall~\{B_{i}\}\in\setB^{n_{1}+n_{2}+n_{3}}. 
  \] 

  Acting similarly, for each $m=1,2,\ldots$ we can specify an integer 
  $n_{m}\ge 1$ and have a set of matrices 
  \[ 
    \bar{A}_{n_{1}+\cdots+n_{m-1}+1},\ldots,\bar{A}_{n_{1}+\cdots+n_{m-1}+n_{m}}, 
  \] 
  for which 
  \begin{equation}\label{E:prodnk-pos3} 
    \max_{1\le k\le n_{1}+\cdots+n_{m}}\|\bar{A}_{k}B_{k}\cdots 
    \bar{A}_{1}B_{1}\|\ge m,\qquad 
    \forall~\{B_{i}\}\in\setB^{n_{1}+\cdots+n_{m}}. 
  \end{equation} 

  The resulting estimate contradicts the assumption from Question~\ref{Q:1} 
  that for every sequence of matrices $\{A_{n}\}\in\setA^{\bbN}$ there is a 
  sequence of matrices $\{B_{n}\}\in\setB^{\bbN}$ for which the sequence of 
  norms $\{\|A_{n}B_{n}\cdots A_{1}B_{1}\|,~n=1,2,\ldots\}$ is bounded. 
  Indeed, due to inequalities~\eqref{E:prodnk-pos3} for the constructed 
  sequence of matrices $\{\bar{A}_{i}\}\in\setA^{\bbN}$ and \textbf{any} 
  sequence of matrices $\{B_{i}\}\in\setB^{\bbN}$ the sequence of norms 
  $\{\|\bar{A}_{n}B_{n}\cdots \bar{A}_{1}B_{1}\|\}_{n=1}^{\infty}$ is 
  unbounded. Theorem~\ref{T2} is proved. 
\end{proof} 

In~\cite[Thm.~4]{Koz:DDNS18} the following statement was made: 
\begin{proposition}\label{St:3} 
  Let the sets of matrices $\setA$ and $\setB$ be such that for every sequence
  of matrices $\{A_{n}\}\in\setA^{\bbN}$ there exists a natural $k$ such that
  for a set of matrices $B_{1},\ldots,B_{k}\in\setB$ the inequality
  $\|A_{k}B_{k}\cdots A_{1}B_{1}\|< 1$ holds. Then there are constants $C > 0$
  and $\lambda\in(0,1)$ such that for every sequence of matrices
  $\{A_{n}\}\in\setA^{\bbN}$ there is a sequence of matrices 
  $\{B_{n}\}\in\setB^{\bbN}$, for which $\|A_{n}B_{n}\cdots A_{1}B_{1}\|\le 
  C\lambda^{n}$ for all $n=1,2,\ldots\,$. 
\end{proposition} 

It follows from this statement that if the matrix sets $\setA$ and $\setB$ 
satisfy the conditions of Proposition~\ref{St:3}, the answer to 
Question~\ref{Q:1} is also positive. Therefore, when analyzing the 
boundedness of matrix products with factors from two sets of matrices, it is 
sufficient to restrict to the case where the conditions of 
Proposition~\ref{St:3} are not satisfied for the sets of matrices $\setA$ and 
$\setB$, i.e., \emph{in the set of matrices $\setA$ there is a sequence of 
matrices $\{A_{n}\}$ such that the inequality $\|A_{k}B_{k}\cdots 
A_{1}B_{1}\|\ge1$ holds for any set of matrices 
$B_{1},\ldots,B_{k}\in\setB$.} 

\section{Pointwise Setting}\label{S:pointwise} 

It is known~\cite{Koz:AiT90:10:e, DaubLag:LAA92, BerWang:LAA92, Gurv:LAA95, 
Hartfiel:02} that Proposition~\ref{St:1} on boundedness of matrix products 
remains valid if one replaces in it the assumption that every sequence of 
norms $\{\|A_{n}\cdots A_{1}\|\}_{n=1}^{\infty}$ is bounded, by the 
assumption that the sequence of norms $\{\|A_{n}\cdots 
A_{1}x\|\}_{n=1}^{\infty}$ is `pointwise bounded', that is, it is bounded for 
every vector $x\in\mathbb{R}^{N}$ and every sequence of matrices 
$\{A_{n}\}\in\setA^{\bbN}$. 

In this context, the question arises of the validity of a similar extension 
for matrix products with alternating factors. We say that a sequence of 
matrices $\{A_{n}\}\in\setA^{\bbN}$ is \emph{pointwise 
$\setB$-right-bounded}, if for every vector $x\in\mathbb{R}^{N}$ and every 
sequence of matrices $\{A_{n}\}\in\setA^{\bbN}$ there is such a constant 
$C=C(x,\{A_{n}\})$ and a sequence of matrices $\{B_{n}\}\in\setB^{\bbN}$ for 
which the following inequality holds: 
\[ 
  \|A_{n}B_{n}\cdots A_{1}B_{1}x\|\le C\|x\|,\qquad n=1,2,\ldots\,. 
\] 

\begin{problem}\label{Q:2} 
  Let $\setA$ and $\setB$ be finite sets of matrices such that every sequence 
  of matrices $\{A_{n}\}\in\setA^{\bbN}$ is pointwise $\setB$ right-bounded. 
  In this case, is there a universal constant $C > 0$ such that for every 
  vector $x\in\mathbb{R}^{N}$ and every sequence of matrices 
  $\{A_{n}\}\in\setA^{\bbN}$ there is a sequence of matrices 
  $\{B_{n}\}\in\setB^{\bbN}$, for which for all $n=1,2,\ldots$ the 
  inequalities $\|A_{n}B_{n}\cdots A_{1}B_{1}x\|\le C\|x\|$ hold? In other 
  words, is it true that if every $\{A_{n}\}\in\setA^{\bbN}$ is pointwise 
  $\setB$-right-bounded, then every sequence of matrices 
  $\{A_{n}\}\in\setA^{\bbN}$ is $\setB$-right-bounded? 
\end{problem} 

Let us show that in the general case the answer to Question~\ref{Q:2} is 
negative. 

\begin{example}\label{Ex:Stan} 
  In~\cite[Example 2]{Stanford:SIAMJCO79} D.P.\,Stanford constructed an 
  example of a set $\setH$ consisting of two matrices 
  \[ 
    H_{1}= 
    \begin{bmatrix} \frac{1}{2} & 0 \\ 0& 2\vphantom{\frac{1}{2}} 
    \end{bmatrix},\quad 
    H_{2}= 
    \begin{bmatrix} \frac{\sqrt{3}}{2} & \frac{1}{2} \\ -\frac{1}{2}& 
      \frac{\sqrt{3}}{2} 
    \end{bmatrix}, 
  \] 
  with the property that for any sequence of matrices 
  $\{A_{n}\}\in\setH^{\bbN}$ the following inequalities hold: 
  \begin{equation}\label{E:AnA1} 
    \|A_{n}\cdots A_{1}\|\ge 1,\qquad n=1,2,\ldots\,, 
  \end{equation} 
  and for each vector $x\in\mathbb{R}^{2}$ there is a sequence of matrices 
  $\{\bar{A}_{n}\}\in\setH^{\bbN}$ such that 
  \begin{equation}\label{E:AnA1x} 
    \bar{A}_{n}\cdots \bar{A}_{1}x\to 0\quad\text{for}\quad n\to\infty. 
  \end{equation} 

  Relation~\eqref{E:AnA1} follows from the inequality 
  \begin{equation}\label{E:pho-det} 
    \|A\|\ge\rho(A)\ge\left|\det A\right|^{1/N}, 
  \end{equation} 
  valid for any $(N\times N)$-matrix $A$, where $\rho(\cdot)$ is the spectral 
  radius of a matrix. 

  For matrices $H_{1}$ and $H_{2}$, the equalities $\det H_{1}=\det H_{2}=1$ 
  are obvious, and hence 
  \[ 
    \det(A_{n}\cdots A_{1})=1. 
  \] 
  From this and from~\eqref{E:pho-det} it follows that 
  \[ 
    \rho(A_{n}\cdots A_{1})\ge\sqrt{\det(A_{n}\cdots A_{1})}=1,\qquad 
    \forall~\{A_{n}\}\in\setH^{\bbN}. 
  \] 
  And then, due to~\eqref{E:pho-det}, 
  \[ 
    \|A_{n}\cdots A_{1}\|\ge \rho(A_{n}\cdots A_{1})\ge1,\qquad 
    n=1,2,\ldots\,. 
  \] 

  The construction (for the chosen vector $x\in\mathbb{R}^{2}$) of a matrix 
  sequence $\{\bar{A}_{n}\}\in\setH^{\bbN}$ for which~\eqref{E:AnA1x} holds 
  is based on the following observation. Since $H_{2}$ is the $30^{\circ}$ 
  rotation matrix, there are always at most $6$ matrices 
  $\bar{A}_{1},\ldots,\bar{A}_{i}=H_{2}$ for the vector $x$ (i.e. $i\le 6$), 
  whose successive application to the vector $x$ under the assumption 
  $\|x\|\le 1$, where $\|\cdot\|$ is the Euclidean norm in $\mathbb{R}^{2}$, 
  leads to to the hit of the vector $\bar{A}_{i}\cdots \bar{A}_{1}x$ in one 
  of the sectors $S$ or $\tilde{S}$ which have a gap $30^{\circ}$ and are 
  symmetric with respect to the abscissa axis, see Fig.~\ref{F:S}. In this 
  case, applying the matrix $\bar{A}_{i+1}=H_{1}$ to the corresponding 
  product $\bar{A}_{i}\cdots \bar{A}_{1}x$, as can be easily calculated, 
  leads to the hit of the vector $x_{1}=\bar{A}_{i+1}\cdot \bar{A}_{i}\cdots 
  \bar{A}_{1}x$ in one of the sets $H_{1}S$ or $H_{1}\tilde{S}$, and for this 
  vector the following estimation holds: 
  \begin{equation}\label{E:normx1} 
    \|x_{1}\|\le q \|x\|, 
  \end{equation} 
  where $q<1$ is a constant, see Fig.~\ref{F:S}. 

  \begin{figure}[!htbp] 
    \center%
    \includegraphics{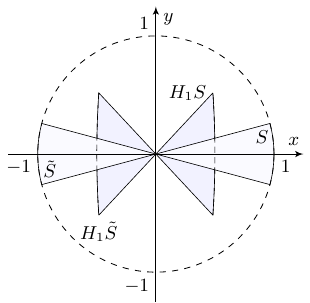} 
    \caption{Illustration for Stanford's example}\label{F:S} 
  \end{figure} 

  Continuing this procedure, we obtain a sequence of vectors $x_{n}\to 0$ 
  satisfying the relations $\|x_{n}\|\le q \|x_{n-1}\|$, from which 
  follows~\eqref{E:AnA1x}. 
\end{example} 

\begin{remark}\label{R:ExtStanford} 
  The described procedure remains in force if instead of the matrices $H_{1}, 
  H_{2}$ the matrices $\alpha H_{1}, \alpha H_{2}$ are taken with the constant
  $\alpha > 1$ being chosen in such a way that the validity of the
  estimation~\eqref{E:normx1} with a constant $q=q_{\alpha} < 1$ is still 
  guaranteed. In this case, as before, for each vector $x\in\mathbb{R}^{2}$ 
  there is a sequence of matrices $\{A_{n}\}\in\setH^{\bbN}$ for which the 
  convergence~\eqref{E:AnA1x} is valid. But since $\det A_{i}=\alpha^{2}$ is 
  now valid, a stronger statement holds instead of~\eqref{E:AnA1}: 
  \[ 
    \|A_{n}\cdots A_{1}\|\ge \rho(A_{n}\cdots A_{1})\ge 
    \sqrt{\det(A_{n}\cdots A_{1})}=\alpha^{n} \to\infty\quad\text{for}\quad 
    n\to\infty. 
  \] 
\end{remark} 

\begin{example}\label{R:Z2} 
  In addition to Example 2 from~\cite{Stanford:SIAMJCO79} quoted above, we 
  note without proof another example of the set $\setH$ for which the 
  relations~\eqref{E:AnA1} and~\eqref{E:AnA1x} hold, 
  see~\cite[p.~50]{Emelyanov:e}. Denote by $\setH$ the set consisting of two 
  matrices 
  \[ 
    H_{1}= 
    \begin{bmatrix} \cos\alpha & -\sin\alpha \\ \sin\alpha& 
      \hphantom{-}\cos\alpha 
    \end{bmatrix},\quad 
    H_{2}= 
    \begin{bmatrix} \cos\beta & 
      -\gamma\sin\beta \\ \gamma^{-1}\sin\beta& \hphantom{-\gamma}\cos\beta 
    \end{bmatrix}, 
  \] 
  where $\gamma>0$, $\gamma\neq 1$ and the quantities $\alpha$ and $\beta$ 
  are incommensurable with $\pi$. 

  In this case, for any sequence of matrices $\{A_{n}\}\in\setH^{\bbN}$ 
  relations~\eqref{E:AnA1} hold, i.e., $\|A_{n}\cdots A_{1}\|\ge 1$ for 
  $n=1,2,\ldots$, and furthermore, for each nonzero vector 
  $x\in\mathbb{R}^{2}$, one can specify a sequence of matrices 
  $\{\bar{A}_{n}\}\in\setH^{\bbN}$ for which convergence~\eqref{E:AnA1x} 
  takes place: $\bar{A}_{n}\cdots \bar{A}_{1}x\to 0$ for $n\to\infty$. 
  Moreover, for any nonzero vector $x\in\mathbb{R}^{2}$, one can also specify 
  a sequence of matrices $\{\tilde{A}_{i}\in\setH\}$ for which the divergence 
  takes place: $\|\tilde{A}_{n}\cdots \tilde{A}_{1}x\|\to\infty$ for 
  $n\to\infty$. 
\end{example} 

Now we can return to Question~\ref{Q:2} and answer it negatively. To do this, 
it suffices to consider the following sets of $(2\times2 )$-matrices: 
\[ 
  \setA:=\{I\},\quad \setB:=\{H_{1}, H_{2}\}, 
\] 
where $H_{1}, H_{2}$ are matrices of Remark~\ref{R:ExtStanford}. As follows 
from Remark~\ref{R:ExtStanford}, in this case any sequence of matrices 
$\{A_{n}\}\in\setA^{\bbN}$ (and such a sequence is unique: $A_{n}\equiv I$) 
is pointwise $\setB$-right-bounded, but not $\setB$-right-bounded, since due 
to Remark~\ref{R:ExtStanford} 
\[ 
  \|A_{n}B_{n}\cdots A_{1}B_{1}\|=\|B_{n}\cdots B_{1}\|\ge \alpha^{n} 
  \to\infty\quad\text{for}\quad n\to\infty, 
\] 
for any sequence of matrices $B_{n}\in\setB=\{H_{1}, H_{2}\}$. 

The example constructed permits the following generalization. 
\begin{theorem}\label{T:nonPWB} 
  Let $\setA$ be a finite set of $(2\times2)$-matrices whose determinants are 
  equal to $1$. Then there is a finite set $\setB$ of nonsingular 
  $(2\times2)$-matrices such that every sequence of matrices 
  $\{A_{n}\}\in\setA^{\bbN}$ is pointwise $\setB$-right-bounded, but not 
  $\setB$-right-bounded. 
\end{theorem} 

\begin{proof} 
  Let us choose arbitrary matrices $H_{1}, H_{2}$, defined in 
  Remark~\ref{R:ExtStanford}. Furthermore, for each matrix $A_{n}\in\setA$ we 
  construct two matrices 
  \[ 
    B_{1,n}=A_{n}^{-1}H_{1},\quad B_{2,n}=A_{n}^{-1}H_{2} 
  \] 
  and define $\setB$ to be the set of all matrices of the form $B_{1,n}, 
  B_{2,n}$. 

  Now note that due to the definition of the matrices $H_{1}, H_{2}$ in 
  Remark~\ref{R:ExtStanford}, 
  \[ 
    \det H_{1}=\det H_{2}=\alpha^{2}>1, 
  \] 
  which means that 
  \[ 
    \det A_{n}=1,\quad \det B_{n}=\alpha^{2} > 1\qquad \forall~ 
    A_{n}\in\setA, B_{n}\in\setB. 
  \] 
  Therefore, 
  \[ 
    \|A_{n}B_{n}\cdots A_{1}B_{1}\|\ge\sqrt{\det(A_{n}B_{n}\cdots 
    A_{1}B_{1})}= \alpha^{n}\to\infty\quad\text{for}\quad n\to\infty, 
  \] 
  for any sequence of matrices $\{A_{n}\}\in\setA^{\bbN}$ and any sequence of 
  matrices $\{B_{n}\}\in\setB^{\bbN}$, from which it follows that no sequence 
  of matrices $\{A_{n}\}\in\setA^{\bbN}$ is $\setB$-right bounded. 

  It remains to show that nevertheless any sequence of matrices 
  $\{A_{n}\}\in\setA^{\bbN}$ is pointwise $\setB$-right-bounded. To do this, 
  we specify any sequence of matrices $\{A_{n}\}\in\setA^{\bbN}$, and a 
  nonzero vector $x\in\mathbb{R}^{2}$. Next, we choose a sequence of matrices 
  $H_{i_{n}}$ that satisfy the relation 
  \begin{equation}\label{E:defHin} 
    H_{i_{n}}\cdots H_{i_{1}}x\to0. 
  \end{equation} 
  This is possible due to Remark~\ref{R:ExtStanford}. Finally, we construct a 
  sequence of matrices $\{B_{n}\}\in\setB^{\bbN}$ by taking 
  \[ 
    B_{n}=B_{i_{n},n}=A_{n}^{-1}H_{i_{n}}. 
  \] 
  With the chosen sequence of matrices $\{B_{n}\}\in\setB^{\bbN}$, the 
  corresponding sequence $A_{n}B_{n}\cdots A_{1}B_{1}x$ takes the following 
  form 
  \[ 
    A_{n}B_{n}\cdots A_{1}B_{1}x =H_{i_{n}}\cdots H_{i_{1}}x, 
  \] 
  whence, due to~\eqref{E:defHin} $A_{n}B_{n}\cdots A_{1}B_{1}x\to0$. The 
  pointwise $\setB$-right-boundedness of the sequence of matrices 
  $\{A_{n}\}\in\setA^{\bbN}$ is proved. 
\end{proof}

\end{document}